\documentclass[a4paper]{article}
\pdfoutput=1
\usepackage{amsmath,amsthm}
\usepackage{amsfonts}
\usepackage{graphicx}
\usepackage{multicol}
\usepackage{color}
\usepackage{amsthm}
\usepackage{enumerate}
\usepackage{amssymb}
\usepackage{mathrsfs}

\usepackage{fancyhdr}
\usepackage[]{graphicx}

\usepackage{amsfonts}
\newcommand\oo{
  \mathchoice
    {{\scriptstyle\mathcal{O}}}
    {{\scriptstyle\mathcal{O}}}
    {{\scriptscriptstyle\mathcal{O}}}
    {\scalebox{.7}{$\scriptscriptstyle\mathcal{O}$}}
  }
  
\DeclareMathOperator{\supp}{supp}
\DeclareMathOperator{\Lpre}{L_{pre}}
\DeclareMathOperator{\Lu}{L_{u}}
\DeclareMathOperator{\conv}{conv}
\DeclareMathOperator{\sign}{sign}
\DeclareMathOperator{\Span}{span}
\DeclareMathOperator{\Dom}{Dom}
\DeclareMathOperator{\card}{card}
\def\O{\mathcal{O}}

\newtheorem{theorem}{Theorem}
\newtheorem{definition}[theorem]{Definition}
\newtheorem{lemma}[theorem]{Lemma}
\newtheorem{proposition}[theorem]{Proposition}
\newtheorem{corollary}[theorem]{Corollary}

\newtheorem{example}{Example}

\newtheorem{remark}{Remark}

\title{\LARGE\textbf{The precontraction group of the field of logarithmic transseries $\mathbb{T}_{\log}$}}
\author{ \textbf{Jos\'e Leonardo \'Angel} }
\date{}

\begin{document}
\fancyhead{}
\fancyhead[LE, RO]{\thepage}
\fancyfoot{}
\renewcommand{\headrulewidth}{0.9pt}
\maketitle

\begin{abstract} \noindent
As a first step to understand the theory of the structure $\mathbb{T}_{\log}$ of logarithmic transseries as an ordered valued logarithmic field, we focus on the map $\chi$ induced by the logarithm of $\mathbb{T}_{\log}$ in its value group $\Gamma_{\log}$ and study the theory of the precontraction group $(\Gamma_{\log},\chi)$. Particularly, we show that this theory is model complete and complete, and we characterize all definable subsets of the discrete set $\chi(\Gamma_{\log})$. 

\end{abstract}

\quad\textbf{Key words}: Precontraction group, centripetal, logarithmic transseries.

\section{Introduction}

In \cite{Kuh4,Kuh6}, Franz-Victor Kuhlmann and Salma Kuhlmann showed that in a non-archimedean exponential field the exponential induces a map, called \textit{contraction}, on the value group of the field with respect to its natural valuation. Specifically, if $\log$ denotes the inverse of the exponential map and $v$ the natural valuation of the ordered field, then for $a>0$ and $v(a)<0$ they defined the contraction map $\chi$ as $\chi(v(a))=v(\log(a))$, $\chi(-v(a))=-\chi(v(a))$ and $\chi(0)=0$. Under this definition, the autors studied in \cite{Kuh2, Kuh4} the first order theory of the value group of an exponential field endowed with such contraction map and showed that this theory is complete, decidable, admits quantifier elimination and is weakly o-minimal. 

We recall that an exponential field is an ordered field equipped with an order preserving group isomorphism from the additive group of the field onto the multiplicative group of positive elements. The transseries field $\mathbb{T}$ is an important non-archimedean exponential field introduced by \'Ecalle in \cite{Ecalle} and by Dahn and Göring in \cite{Danh}, and widely studied as a valued diferential field in \cite{Van6} by Matthias Aschenbrenner, Lou van den Dries and Joris van der Hoeven. Particularly, the last authors show that the contraction map associated to the exponential map of $\mathbb{T}$ is definable in the asymptotic couple of $\mathbb{T}$, that is the structure of the value group of $\mathbb{T}$ endowed with a function induced by the diferential map. In similar way, in \cite{Ger1} Allen Geheret shows that for the valued diferential field of logarithmic transseries $\mathbb{T}_{\log}$, a special sufield of $\mathbb{T}$ defined in \cite{Van6} and whose elements, informally speaking, are formal series which do not involve exponentiation, there is a precontraction map (i.e a non-surjective contraction map) definable in the asymptotic couple of $\mathbb{T}_{\log}$.

Following the classical strategy used in model theory to study the theory of a valued field by first studying the theory of its value group and of its residual field, as a first step to understand the theory of $\mathbb{T}_{\log}$ as an ordered valued logarithmic field, i.e an ordered valued field equipped with an order preserving group morphism from the multiplicative group of positive elements of the field in the additive group. We study in this paper  the model theory of its associated precontraction group, that is the structure given by its value group $\Gamma_{\log}$ endowed with a function $\chi$ induced by the logarithm map.

Base on the ideas used in \cite{Kuh2, Kuh4} to study the theory of the contraction groups and those used in \cite{Ger1} to study the theory of the asymptotic couple of $\mathbb{T}_{\log}$, we study the first order theory of the couple $(\Gamma_{\log},\chi)$ as a precontraction group. We notice that although the map $\chi$ is not surjective, the image of $\Gamma_{\log}^{<0}$ by $\chi$ is a discrete set cofinal in $\Gamma_{\log}^{<0}$ and using this fact we prove that the theory of the precontraction group  $(\Gamma_{\log},\chi)$ is model complete and complete and we study the definable subsets of the image of $\Gamma_{\log}^{<0}$ by $\chi$.

The structure of the paper is as follows. In section 2, we recall some preliminary notions and notations about ordered abelian groups and valued abelian groups and we present a short description of $\mathbb{T}_{\log}$ and its value group. In section 3, we include some definitions and results about precontraction groups. In section 4, we define the language $L_{pdg}$, of ordered groups together a symbol function for the contraction  map and a constant symbol, and study the $L_{pdg}$-theory $T_{pdg}$ of centripetal precontraction discrete groups. Particularly, we prove that the theory $T_{pdg}$ is model complete and complete. Next, we expand the language $L_{pdg}$ to ensure that the natural expansion of the theory $T_{pdg}$ has quantifier elimination and use it to characterize all definable subsets of the image of the group by the precontraction map. Finally, we study the simple extensions of models of $T_{pdg}$.

For the general notions and facts about model theory, we refer the reader to \cite{Chang, Hod} or \cite[Appendix B]{Van6}. 


\section{Preliminaries}
Throughout, $m$ and $n$ range over $\mathbb{N}=\lbrace 0,1,2, ...\rbrace$, the set of natural numbers.
\subsubsection*{Ordered sets}

By an ordered set $S$ we mean a set $S$ equipped with a distinguished total order relation $\leq$. If $B$ is a subset of $S$ we see $B$ as an ordered subset of $S$ ordered by the induced ordering and we define the set 
\[S^{\geq B}=\lbrace a\in S: a\geq b \text{ for all }b\in B\rbrace.\]
In similar way we define $S^{>B}$, $S^{\leq B}$ and $S^{<B}$. Particularly, if $B=\lbrace b\rbrace$, then we set $S^{\geq b}=S^{\geq B}$.

We say that a subset $B$ of $S$ is convex in $S$ if for all $a,c\in B$ and $b\in S$ such that $a\leq b\leq c$ we have $b\in B$, and we define the \textit{convex hull }of $B$ in $S$ as 
\[\conv(B)=\lbrace b\in S: a\leq b\leq c \text{ for some } a,b\in B\rbrace.\]
Moreover, we say that $B\subseteq S$ is a \textit{lower cut} in $S$ if for all $b\in B$ and $a\in A$, $a<b$ implies $a\in B$.

Finally, we define intervals in $S$ as usual and for $\infty\notin S$, we define the set $S_{\infty}=S\cup \lbrace\infty \rbrace$ and extend the order of $S$ to $S_{\infty}$ setting $a<\infty$ for all $a\in S$.

\subsubsection*{Ordered abelian groups} 
An ordered abelian group $\Gamma$, written additively, is an abelian group with an ordering such that for all $a,b,c\in \Gamma$ if $a<b$ then $a+c<b+c$. 
For $a\in \Gamma$ we set $|a|=\max\lbrace a, -a\rbrace$ and define the \textit{archimedean class} of $a$ in $\Gamma$ as 
\[[a]=\lbrace b\in \Gamma: |a|\leq n|b| \text{ and } |b|\leq n|a| \text{ for some } n\geq 1\rbrace.\]
Thus, we say that $a$ is \textit{archimedean equivalent} to $b$ in $\Gamma$ if $b\in [a]$. Moreover, the set $[\Gamma]$ of all archimedean classes become in an ordered set putting
\[[a]\leq [b] \Leftrightarrow |a|\leq n|b| \text{ for some } n\geq 1.\]
Moreover, we have 
\[[a]< [b] \Leftrightarrow n|a|\leq |b| \text{ for all } n\geq 1.\]
\subsubsection*{Valued abelian groups}
Let $\Gamma$ be an abelian group and $S$ be an ordered set. A \textit{valuation}  on $\Gamma$ is a surjective map $v:\Gamma\rightarrow S_{\infty}$ such that for all $a, b\in \Gamma$ the following conditions are satisfied:
\begin{enumerate}
\item [$(1)$] $v(a)=\infty \Leftrightarrow a=0$.
\item [$(2)$] $v(-a)=v(a)$.
\item [$(3)$] $v(a+b)\geq \min\lbrace v(a), v(b)\rbrace$.
\end{enumerate}

A \textit{valued abelian group} is a structure conformed by an abelian group $\Gamma$, and ordered set $S$ and a valuation $v$ on $\Gamma$.

For example, for an ordered abelian group $\Gamma$ if we put $S=[\Gamma]$ and equip $S$ with the reversed ordering of $[\Gamma]$, then the map $v:\Gamma\rightarrow S$ defined as $v(a)=[a]$ is a valuation on $\Gamma$. We call this valuation the \textit{natural valuation} of $\Gamma$.

\subsubsection*{The field of logarithmic transseries $\mathbb{T}_{\log}$}

The field $\mathbb{T}_{\log}$ of logarithmic transseries is a special subfield of the field $\mathbb{T}$ of transseries (see \cite{Van6} for a definition of $\mathbb{T}$), in which each element is a formal series with real coeficients and monomials of the form $\ell_0^{r_0}\ell_1^{r_1}\cdots\ell_n^{r_n}$, with $\ell_0=x$, $\ell_{n+1}=\log (\ell_n)$ for $n>0$ and $r_0,...,r_n\in \mathbb{R}$. 

Formally, we can construct $\mathbb{T}_{\log}$ as follows: First, for each $n$ we set $\mathcal{L}_n$ as the formal multiplicative group given by
\[\mathcal{L}_n=\lbrace \ell_0^{r_0}\ell_1^{r_1}\cdots\ell_n^{r_n}: r_0,r_1,...,r_n \in \mathbb{R}\rbrace,\]
and ordered by the relation $\ell_0^{r_0}\ell_1^{r_1}\cdots\ell_n^{r_n}>1$ if and only if the exponents $r_0,r_1,...,r_n$ are not all zero, and $r_i>0$ for the least $i$ with $r_i\neq 0$. 

Next, for each $n$, we define the Hahn field $\mathbb{R}[[\mathcal{L}_n]]$ of well based series with real coefficients and monomials in $\mathcal{L}_n$.  We mean the field of all functions $f:\mathcal{L}_n\rightarrow \mathbb{R}$ (written as formal sums $f=\sum\limits_{m\in \mathcal{L}_n} f_m m$) such that $\supp(f):=\lbrace m \in \mathcal{L}_n:f_m\neq 0 \rbrace$ has no strictly increasing infinite sequences.

Finally, since $\mathcal{L}_m$ is an ordered subgroup of $\mathcal{L}_n$ for $m\leq n$, the ordered group inclusions 
\[\mathcal{L}_0\subseteq \mathcal{L}_1\subseteq...\subseteq \mathcal{L}=\bigcup\limits_{n}\mathcal{L}_n,\]
induce field inclusions
\[\mathbb{R}[[x^{\mathbb{R}}]]=\mathbb{R}[[\mathcal{L}_0]]\subseteq \mathbb{R}[[\mathcal{L}_1]]\subseteq...\]
and we define
\[\mathbb{T}_{\log}:=\bigcup\limits_{n}\mathbb{R}[[\mathcal{L}_n]].\]
\noindent
It follows that $\mathbb{T}_{\log}$ is an ordered subfield of $\mathbb{T}$ and $\mathbb{R}[[\mathcal{L}]]\cap \mathbb{T}=\mathbb{T}_{\log}$. Moreover, as each group $\mathcal{L}_n$ is divisible, the fields $\mathbb{R}[[\mathcal{L}_n]]$ and $\mathbb{T}_{\log}$ are real closed.\\

\noindent
Let $\Gamma_{\log}$ be the ordered $\mathbb{R}$-vector space $\bigoplus\limits_{n>0}\mathbb{R}\ell_n$, where $\alpha=\sum\limits_{i=0}^nr_i\ell_{i+1}>0$ if $r_k>0$ for the least $k$ in $\lbrace0,1,...,n\rbrace$ such that $r_k\neq 0$. We define a convex valuation $v$ of $\mathbb{T}_{\log}$ as the unique map $$v:\mathbb{T}_{\log}\rightarrow \Gamma_{\log}\cup \lbrace \infty\rbrace$$  
\noindent
such that 
\begin{enumerate}
\item [$(1)$] $v(\ell_0^{r_0}\ell_1^{r_1}\cdots\ell_n^{r_n})=-r_0\ell_1-r_1\ell_2-\cdots-r_n\ell_{n+1}$,
\item [$(2)$] $v(f)=v(\frak{d}(f))$ for all $f\in \mathbb{T}_{log}^{\neq 0}$, where $\mathfrak{d}(f):=\max(\supp(f))$ is the dominant monomial of $f$. 
\item  [$(3)$] $v(0)=\infty$.
\end{enumerate}
Thus, $\mathbb{T}_{\log}$ becomes an ordered valued field with valuation ring  $\O_{\log}=\mathbb{R}\oplus \oo_{\log}$, maximal ideal 
\[\oo_{\log}=\lbrace f\in \mathbb{T}_{log}:v(f)>0\rbrace,\]
value group $\Gamma_{\log}$, and residue field $\mathbb{R}$.

Now, since each positive element $f\in \mathbb{T}_{\log}$ can be decomposed as $f=\frak{d}(f)\cdot f_{\frak{d}(f)}\cdot (1+\epsilon)$ where $\frak{d}(f)\in \mathcal{L}$, $f_{\frak{d}(f)}\in \mathbb{R}^{>0}$ is the leading coefficient of $f$ and $\epsilon\in \oo_{\log}$ (see \cite{Kuh4}), we may define the logarithm of $f$ as 
\[\log(f)=\Lpre(\frak{d}(f))+\log_{\mathbb{R}}(f_{\frak{d}(f)})+\Lu(1+\epsilon),\]
where $\log_{\mathbb{R}}$ is the logarithm in $\mathbb{R}$, $\Lu$ is the logarithm on 1-units given by
\[\Lu(1+\epsilon)=\sum\limits_{i>0}(-1)^{i+1}\dfrac{\epsilon^i}{i}\]
and $\Lpre$ is the logarithmic section defined as $\Lpre(\ell_0^{r_0}\ell_1^{r_1}\cdots\ell_n^{r_n})=r_0\ell_1+\cdots+r_n\ell_{n+1}$.

Under this definition we see that the map $\log$ is an ordered embedding from the multiplicative group $\mathbb{T}_{\log}^{>0}$ into the additive group $\mathbb{T}_{log}$, such that
\[\log(\mathbb{T}_{\log}^{>0})=\Gamma_{\log}\oplus \mathbb{R}\oplus \oo_{\log}\]
is an $\mathbb{R}$-vector subspace of $\mathbb{T}_{\log}$. 

Additionally, the valuation and the logarithm are related by the following property, which is known as \textit{Growth Axiom}(see \cite{Kuh6}): for all $f\in \mathbb{T}_{\log}^{>0}$ with $v(f)<0$ we have that $v(\log(f))>v(f)$, which implies 

$$f>\log(f^n)=n\log(f) \text{ for all } n\in \mathbb{N}.$$

\noindent
Moreover, the map $\log$ induce an extra structure in the value group $\Gamma_{\log}$ given by the map $$\chi:\Gamma_{\log}\rightarrow \Gamma_{\log}$$ 
\noindent
defined as
\[\chi(\alpha)=
\begin{cases}
\chi'(\alpha), & \text{if $\alpha<0$},\\
0, & \text{if $\alpha=0$},\\
-\chi'(-\alpha), & \text{if $\alpha>0$.}\\
\end{cases}\]
where $\chi':\Gamma_{\log}^{<0}\rightarrow \Gamma_{\log}^{<0}$ is given by $\chi'(\alpha)=v(\log(f))$ with $f\in \mathbb{T}_{\log}^{>0}$, and $\alpha=v(f)<0$. We see that $\chi$ is well defined, since for $f,g\in \mathbb{T}_{\log}^{>0}$ with $v(f)=v(g)<0$ there is a positive unit $h$ in $\O_{\log} $such that $f=gh$. Thus, $\log(f)=\log(g)+\log(h)$ and 
\[v(\log(f))\geq \min\lbrace v(\log(g)),v(\log(h))\rbrace.\]
By definition of $\chi$ we have $v(\log(h))\geq 0$ and $v(\log(g))<0$, and then 
\[v(\log(f))= \min\lbrace v(\log(g)),v(\log(h))\rbrace=v(\log(g)).\]

\section{Precontraction groups}

The notion of contraction map was used in \cite{Kuh4} to study the structure of the value group of an exponential field and the theory of contraction groups was studied in \cite{Kuh2,Kuh3}. We list here some useful definitions and results of those papers. Specifically:\\

\begin{definition}
Given a totally ordered abelian group $\Gamma$ and a map $\chi:\Gamma \rightarrow \Gamma$, the pair $(\Gamma,\chi)$ is called a \textit{precontraction group} and $\chi$ is called a \textit{precontraction map} if it satisfies for all $a, b\in \Gamma$ the following axioms:

\begin{enumerate}
\item [$(1)$] $\chi(a)=0 \leftrightarrow a=0$,
\item [$(2)$] $a\leq b \rightarrow \chi(a)\leq \chi(b)$,
\item [$(3)$] $ \chi(-a)=-\chi(a)$,
\item [$(4)$] if $a$ is archimedean equivalent to $b$ and $\sign(a)=\sign(b)$, then $\chi(a)=\chi(b)$.
\end{enumerate}
If in addition $\chi$ is surjective then $\chi$ is called a \textit{contraction map} and $(\Gamma,\chi)$ is called a \textit{contraction group}. Moreover, $(\Gamma, \chi)$ will be called \textit{centripetal} if $\forall a\in \Gamma^{\neq 0} (|a|>|\chi(a)|)$ and \textit{divisible} if $\Gamma$ is divisible.\\
\end{definition}

\begin{example}\label{ejemlog} The map $\chi$ defined in the value group $\Gamma_{\log}$ of $\mathbb{T}_{\log}$ is a precontraction map. Moreover, since the ordered valued logarithmic field $\mathbb{T}_{\log}$ satisfies the Growth Axiom, then  in fact $(\Gamma_{\log},\chi)$ is a centripetal precontraction group.

\begin{proof} We already see that $\chi$ is well defined. Now, let $v(f)$ be archimedean equivalent to $v(g)$ with $f,g\in \mathbb{T}_{\log}^{>0}$  and $v(f)\leq v(g)<0$, then there is a natural number $n$ such that $nv(g)=v(g^n)\leq v(f)$. By convexity of $v$ we obtain that $g^n\geq f\geq g$, and then $\log(g^n)=n\log(g)\geq \log(f)\geq \log(g)$. Thus $v(\log(g))=v(\log(f))$ and $\chi(v(f))=\chi(v(g))$.

Finally, if $v(f)<0$ with $f\in \mathbb{T}_{\log}^{>0}$, then by Growth Axiom we have $v(f)<v(\log(f))=\chi(v(f)).$ Thus, by definition of $\chi$ we conclude that $|v(a)|>|\chi(a)|$ for all $a\in \Gamma^{\neq 0}$, i.e. $(\Gamma_{\log}, \chi)$ is centripetal.
\end{proof}
\end{example}
\noindent
From the axioms we have some useful consequences:\\

\begin{lemma} \label{cont1}

Let $(\Gamma,\chi)$ be a precontraction group and $a,b\in \Gamma$.

\begin{enumerate}
\item [$(1)$] Axiom $(4)$ is equivalent to the single statement  $\chi(2a)=\chi(a)$.
\item [$(2)$] $\chi(\Gamma^{<0})\subseteq \Gamma^{<0}$ and $\chi(\Gamma^{<0}) =-\chi(\Gamma^{>0})$.
\item [$(3)$] $\chi(a+b)\geq \min\lbrace \chi(a),\chi(b)\rbrace$.
\item [$(4)$] If $\chi(a)<\chi(b)<0$ then $\chi(a-b)=\chi(a)$.
\item [$(5)$] If $0<\chi(a)<\chi(b)$ then $\chi(b-a)=\chi(b)$.
\item [$(6)$] Let $b>0>a$. If $\chi(|a|)>\chi(|b)|$ then  $\chi(a-b)=\chi(a)$, and if $\chi(|b|)>\chi(|a|)$ then $\chi(b-a)=\chi(b)$.

\end{enumerate}
\end{lemma}
\begin{proof}
\begin{enumerate}
\item [$(1)$] We just have to show that the statement $\forall a\in \Gamma$  $\chi(2a)=\chi(a)$ implies axiom $(4)$. First, by axiom $(2)$ we can observe that if $\chi(2a)=\chi(a)$ then $\chi(na)=\chi(a)$ for all $n\in \mathbb{N}$ . Now, If $a$ is archimedean equivalent to $b$ and $\sign(a)=\sign(b)$ then there is a natural number $n$ such that $n|a|\geq |b|$ and $n|b|\geq |a|$, so $\chi(a)=\chi(na)\geq \chi(b)=\chi(nb)\geq\chi(a)$ and thus $\chi(a)=\chi(b)$.
\item [$(2)$] If $a<0$ then by axioms 1 and 2 we have $\chi(a)<0$ and by axiom $(3)$ we have $\chi(-a)=-\chi(a)>0$.
\item [$(3)$] Without loss of generality we can assume that $a<b$. Then 
\[a+a<a+b<b+b\]
and
\[\chi(2a)\leq\chi(a+b)\leq\chi(2b).\]
Since  $\chi(2x)=\chi(x)$ for all $x\in \Gamma$, then 
\[\chi(a)\leq\chi(a+b)\leq\chi(b),\]
so $\chi(a+b)\geq \min\lbrace \chi(a),\chi(b)\rbrace$.
\item [$(4)$] Since $\chi(a)<\chi(b)<0$ then $a<b<0$ and $a-b>a$. Thus $\chi(a-b)\geq\chi(a)$. On the other hand, as $\chi(b)>\chi(a)$ and by item $(3)$ 
\[\chi(a)=\chi((a-b)+b)\geq \min\lbrace \chi(a-b),\chi(b)\rbrace,\]
then  $\chi(a)\geq\chi(a-b)$. Thus, $\chi(a-b)=\chi(a)$.\\

Items $(5)$ and $(6)$ follow of item $(4)$.

\end{enumerate}

\end{proof} 
\noindent
Working with the natural valuation of $\Gamma$, for example we have the following immediate properties:\\

\begin{lemma}\label{cont2} Let $(\Gamma,\chi)$ be a precontraction group. Then
\begin{enumerate}
\item [$(1)$] For all $a,b\in \Gamma$, if $v(a)\leq v(b)$ then $|\chi(a)|\geq |\chi(b)|$. 
\item [$(2)$] For all $a,b\in \Gamma$, if $v(a-b)>v(a)$ then $\chi(a)=\chi(b)$.
\item [$(3)$] For all $a_1,a_2,...,a_n\in \Gamma$, if $v(a_k)<v(a_i)$ for all $i\neq k$ then $\chi(\sum\limits_{i=1}^n a_i)=\chi(a_k)$.
\item [$(4)$] $(\Gamma,\chi)$ is a centripetal precontraction group if and only if $v(\chi(a))>v(a)$ for all $a\in \Gamma^{\neq 0}$.
\end{enumerate}
\end{lemma}

\noindent
Now, the main result about contraction groups proved in \cite{Kuh2, Kuh3} is the following:\\

\begin{theorem} In the language of ordered groups expanded by a unary function symbol for the contraction map, the theory of nontrivial divisible centripetal contraction groups is complete, decidable, admits quantifier elimination and is weakly o-minimal\footnote{A theory in which an order is given or definable is called \textit{weakly o-minimal} if in every model of this theory, each definable subset is a finite union of convex subsets. Moreover, if each one of such convex subsets is an interval, then we say that the theory is \textit{o-minimal}}, and it is the model completion of the theory of centripetal precontraction groups.
\end{theorem}

\section{The theory $T_{pdg}$}\label{sectcd}

A key feature of the centripetal precontraction group $(\Gamma_{\log},\chi)$ is that the image of $\Gamma_{\log}^{<0}$ by $\chi$ is a discrete set with first element and where the immediate successor of an element $a\in \chi(\Gamma_{\log}^{<0})$ is $\chi(a)$. Thus, to capture this property we introduce the following definition:\\
 
\begin{definition}\label{clog}
Let $L_{pdg}=\lbrace +,-,0,<,\chi, c\rbrace,$ be the language of ordered groups augmented by a unary function symbol $\chi$ and a constant symbol $c$. We say that a nontrivial centripetal precontraction group $(\Gamma,\chi)$ is a model of the $L_{pdg}$-theory $T_{pdg}$ if:

\begin{enumerate}
\item [$(1)$] $\chi(\Gamma^{<0})$ has a least element $c$,
\item [$(2)$] $\chi:\chi(\Gamma^{<0})\rightarrow \chi(\Gamma^{<0})^{>c}$ is a bijection,
\item [$(3)$] $\forall a,b \in \chi(\Gamma^{<0})$ if $a<b$ then $a<\chi(a)\leq b$ 
\item [$(4)$] $\Gamma$ is a divisible ordered group.
\end{enumerate}
\end{definition}

\noindent
From the above definition, we can see that each substructure $S$ of a model of $T_{pdg}$ is a centripetal precontraction group where $\chi(S^{<0})$ has a least element and $\chi(a)$ is the immediate successor of $a$ for each $a\in\chi(S^{<0})$.\\

\begin{example} 
\begin{enumerate}

\item [$(1)$] Clearly,  $(\Gamma_{\log},\chi)$ is a model of $T_{pdg}$. Moreover, 
\[\chi(\Gamma_{\log}^{<0})=\lbrace -\ell_2,-\ell_3,... \rbrace,\]
where $\chi(-\ell_k)=-\ell_{k+1}$ and $-\ell_k<-\ell_{k+1}$. 

\item [$(2)$] Let $\oplus_{i}\mathbb{Q}e_i$ be a vector space over $\mathbb{Q}$ with ordered basis $(e_i)$. Under the usual lexicographic order, i.e. 
\[\sum a_ie_i>0\text{ iff } a_k>0 \text{ for the least } k \text{ such that } a_k\neq 0,\]
$\oplus_{i}\mathbb{Q}e_i$ becomes an ordered abelian group and if we define the function $\chi_{\mathbb{Q}}:\oplus_{i}\mathbb{Q}e_i\rightarrow \oplus_{i}\mathbb{Q}e_i$ as 
\[\chi_{\mathbb{Q}}(\sum a_ie_i)=\sign(a_k)e_{k+1}\]
with $k$ the minimal index such that $a_k\neq 0$, then  $(\oplus_{i}\mathbb{Q}e_i,\chi_{\mathbb{Q}})$ is a model of $T_{pdg}$.\\

\end{enumerate}
\end{example}
\noindent
In addition to the properties listed in lemmas \ref{cont1} and \ref{cont2}, we can observe that if $(\Gamma,\chi)$ is a model of $T_{pdg}$, then the discrete set $\chi(\Gamma^{<0})$ is cofinal in $\Gamma^{<0}$ since for all $a\in \Gamma^{<0}$ we have $a<\chi(a)$. Now, although in the models of $T_{pdg}$ the map $\chi$ is not surjective and we can not proceed as in \cite{Kuh2} to prove the model completeness of $T_{pdg}$, here we will use the properties of the discrete set $\chi(\Gamma^{<0})$ to do that. 

\subsection{Some algebraic properties of models of $T_{pdg}$}

First, we can observe the following:\\

\begin{lemma}\label{indep} If $(\Gamma, \chi)$ is a model of $T_{pdg}$, then $\Gamma$ is a vector space over $\mathbb{Q}$ and $\chi(\Gamma^{<0})$ is a linearly independent subset of $\Gamma$.
\end{lemma}
\begin{proof}
Since $\Gamma$ is a divisible ordered group it follows that $\Gamma$ is a vector space over $\mathbb{Q}$. Moreover, given $q_1, q_2,...,q_n\in \mathbb{Q}$ with $q_1\neq 0$ and $a_1,a_2,...,a_n\in \chi(\Gamma^{<0})$ with $a_1<a_2<...<a_n$ then 
\[\chi(a_1)<\chi(a_2)<...<\chi(a_n),\] 
and if $\alpha=\sum\limits_{i=1}^n q_ia_i$ then by lemma \ref{cont1} we have that $\chi(\alpha)=\chi(a_1)$ whenever $q_1>0$ and $\chi(\alpha)=-\chi(a_1)$ whenever $q_1<0$. Thus $\alpha\neq 0$.

\end{proof}
\noindent
Regarding the construction of new precontraction groups we have the following:\\

\begin{lemma} Let $(\Gamma,\chi)$ be a centripetal precontraction group and $\Delta\subseteq \Gamma$ be a nonempty convex subgroup such that if $\chi(x)\in \Delta$ then $x\in \Delta$ for $x\in \Gamma$. Then:
\begin{enumerate}
\item [$(1)$] There is a unique order $\leq'$ in $\Gamma/\Delta$ has a such that $\Gamma/\Delta$ is an ordered abelian group in which if $a\leq b$ then $\overline{a}\leq'\overline{b}$ for $a,b\in \Gamma$.
\item [$(2)$] The map $\chi':\Gamma/\Delta\rightarrow \Gamma/\Delta$ given by $\chi'(\overline{a})=\overline{\chi(a)}$ is well defined and makes $(\Gamma/\Delta,\chi')$ a centripetal precontraction group.
\end{enumerate}

\end{lemma}

\begin{proof} Since the $(1)$ is a general property of ordered abelian groups, it is enough to put $\overline{a}>0$ if and only if $a>\Delta$.
First we show that $\chi'$ is well defined. To do that we prove that if $a-b\in \Delta$ then $\chi(a)-\chi(b)\in \Delta$. If $\chi(a)=\chi(b)$ then clearly $\chi(a)-\chi(b)=0\in \Delta$. Now, if $\chi(a)\neq \chi(b)$ then we have the following cases:
\begin{itemize}
\item $\chi(a)>\chi(b)>0$. Thus, $a-b>0$ and by centripetal property we have $a-b>\chi(a-b)>0$. Moreover, $\chi(a-b)=\chi(a)$. So, $a-b>\chi(a)>\chi(b)>0$ which implies $\chi(a),\chi(b)\in \Delta$ and then $\chi(a)-\chi(b)$. 
\item $\chi(a)<\chi(b)<0$. Similar to the previous case.
\item $a<0<b$. Thus $a<\chi(a)<0<\chi(b)<b$, $a-b<0$ and $a-b<\chi(a-b)<0$. Moreover, $a-b<\chi(a)-\chi(b)<0$ and then $\chi(a)-\chi(b)\in \Delta$.
\end{itemize}
Now, since $\chi(x)\in \Delta$ implies that $x\in \Delta$, then we can prove that $(\Gamma/\Delta,\chi')$ is a centripetal precontraction group.
\end{proof}

\subsection{Embedding lemmas}
Let $(\Gamma,\chi), (\Gamma',\chi')$ be precontraction groups. We say that $(\Gamma',\chi')$ is an extension of $(\Gamma,\chi)$ if $\Gamma'$ is an extension of $\Gamma$ as valued groups $\chi(\Gamma)=\chi'(\Gamma')\cap \Gamma$ and $\chi'(a)=\chi(a)$ for $a\in \Gamma$. Moreover, we say that 
\[\phi:(\Gamma,\chi)\rightarrow (\Gamma',\chi')\]
is an embedding of precontraction groups if $\phi:\Gamma\rightarrow \Gamma'$ is an embedding of ordered abelian groups such that
\[\phi(\chi(a))=\chi'(\phi(a)) \text{ for all } a\in \Gamma.\]
From this definition it follows that if $(\Gamma,\chi)\subseteq (\Gamma',\chi')$ are models of $T_{pdg}$, then we have the following possibilities: First we can have $\chi(\Gamma^{<0})=\chi'(\Gamma'^{<0})$, which is always true if $[\Gamma]=[\Gamma']$ and some times when $[\Gamma]\neq [\Gamma']$. Secondly, we can have $\chi(\Gamma^{<0})\neq \chi'(\Gamma'^{<0})$, and here we have again two possibilities: either there is $b\in\chi'(\Gamma'^{<0})$ such that $b>\chi(\Gamma^{<0})$ or there is a nonempty lower cut $G$ in $\chi(\Gamma^{<0})$ and $b\in \chi'(\Gamma'^{<0})\setminus \chi(\Gamma^{<0})$ such that $\chi(G)\subseteq G$ and $G<b<\Gamma^{>G}$.

\begin{definition}\label{scut} From now on, we call $G\subseteq \chi(\Gamma^{<0})$ a \textit{special cut} if $G$ is a lower cut in $\chi(\Gamma^{<0})$ such that $\chi(G)\subseteq G$ and we denote by $scut(\chi(\Gamma^{<0}))$ the collection of all special cuts of $\chi(\Gamma^{<0})$.  
\end{definition}

\noindent
Based on Gehret's work about the theory of the asymptotic couple of $\mathbb{T}_{\log}$ in \cite{Ger1},  in the following we present some embedding lemmas which deal with the above cases and that will be used to prove the model completeness of $T_{pdg}$.

\subsubsection*{Case 1. $(\Gamma,\chi)\subseteq (\Gamma',\chi')$ with $[\Gamma]=[\Gamma']$.}

From \cite[lemma 3.6]{Kuh2} we have the following result:\\ 
 
\begin{lemma}\label{kul1} Let $(\Gamma,\chi)$ be a centripetal precontraction group. Then for each extension $(\Gamma',<)$ of $(\Gamma,<)$ such that $[\Gamma]=[\Gamma']$, $\chi$ extends in a unique way to a centripetal precontraction $\chi'$ on $\Gamma'$ and we have $\chi(\Gamma')=\chi(\Gamma)$. Particularly, if $\mathbb{Q}\Gamma=\mathbb{Q}\otimes_{\mathbb{Z}}\Gamma$ is the divisible hull of $\Gamma$ then $\chi(\mathbb{Q}\Gamma)=\chi(\Gamma)$, since every element in $\mathbb{Q}\Gamma$ is archimedean equivalent to some element of $\Gamma$.
\end{lemma}

\noindent
Using the quantifier elimination of the theory of divisible ordered abelian groups we have:\\

\begin{lemma} \label{case1}
Let $(\Gamma,\chi), (\Gamma',\chi')$ and $(\Gamma^*,\chi^*)$ be models of $T_{pdg}$, such that $(\Gamma,\chi)\subseteq (\Gamma',\chi')$, $[\Gamma]=[\Gamma']$, $(\Gamma^*,\chi^*)$ is $k$-saturated for some $k>\card(\Gamma')$, and $\phi:(\Gamma,\chi)\rightarrow (\Gamma^*,\chi^*)$ is an embedding, then there is an embedding $\phi': (\Gamma',\chi')\rightarrow (\Gamma^*,\chi^*)$ which extends $\phi$.

\end{lemma}

\begin{proof}
Since $\Gamma, \Gamma'$ and $\Gamma^*$ are divisible ordered abelian groups and such theory has quantifier elimination, then by saturation of  $(\Gamma^*,\chi^*)$ there is an embedding $\phi':\Gamma'\rightarrow \Gamma^*$ that extends the embedding $\phi:\Gamma\rightarrow \Gamma^*$. Moreover, if $b\in \Gamma'$, because $[\Gamma]=[\Gamma']$, there is $a\in \Gamma$ such that $[a]=[b]$ and $\sign(a)=\sign(b)$. Thus,  $\chi'(b)=\chi'(a)$ and then 
\[\phi'(\chi'(b))=\phi'(\chi'(a))=\phi(\chi(a)),\]
but as $[\phi'(b)]=[\phi(a)]$ in $[\Gamma']$, then $\chi^*(\phi'(b))=\chi^*(\phi(a))$. Finally, since $\phi$ is an embedding of centripetal divisible precontraction groups, then 
\[\chi^*(\phi(a))=\phi(\chi(a))=\phi'(\chi'(b)).\]
\end{proof}

\subsubsection*{Case 2. $(\Gamma,\chi)\subseteq (\Gamma',\chi')$ with $[\Gamma]\neq [\Gamma']$ and $\chi(\Gamma^{<0})=\chi'(\Gamma'^{<0})$.}
From \cite[lemma 3.3]{Kuh2} we know that:\\

\begin{lemma}\label{kul2}
Let $(\Gamma,\chi)\subseteq (\Gamma',\chi')$ be precontraction groups. Let $a\in \Gamma'$ such that $[a]\notin [\Gamma]$ and $\chi'(a)=b\in \Gamma$. Then $(\Gamma+\mathbb{Z}a,\chi_{a})$ is a precontraction group with $\chi_{a}(\Gamma+\mathbb{Z}a)=\chi(\Gamma)\cup\lbrace b, -b\rbrace\subset G$. Moreover, the extension of $\chi$ from $(\Gamma,\chi)$ to $\Gamma+\mathbb{Z}a$ is uniquely determined by the assignment $\chi'(a)=b$.\\
\end{lemma}
\noindent
If $\Gamma$ is divisible, $\Gamma+\mathbb{Q}a$ is the divisible hull of $\Gamma+\mathbb{Z}a$. Thus, by lemmas \ref{kul1} and \ref{kul2}  we have $[\Gamma+\mathbb{Q}a]=[\Gamma+\mathbb{Z}a]$ and the image under $\chi_a$ coincide. From this we have the following lemma:\\

\begin{lemma} \label{case2} Let $(\Gamma,\chi)\subseteq (\Gamma',\chi')$ be models of $T_{pdg}$ with $\chi(\Gamma^{<0})=\chi'(\Gamma'^{<0})$,  $a\in\Gamma'^{<0}$ such that $[a]\notin [\Gamma]$ and $C$ is the lower cut in $[\Gamma]$ defined by $[a]$. Then there is a model $(\Delta,\chi_{\Delta})$ of $T_{pdg}$ such that:
\begin{enumerate}
\item [$(1)$] $(\Gamma,\chi)\subset (\Delta,\chi_{\Delta}) \subseteq (\Gamma',\chi')$ with $[a]\in [\Gamma_{\Delta}]$, and
\item [$(2)$] for any embedding $\phi$ of $(\Gamma,\chi)$ into a model $(\Gamma^*,\chi^*)$ of $T_{pdg}$ and each $a'\in \Gamma^{*<0}$ with  $[a']\notin [\phi(\Gamma)]$ which realize the cut $\lbrace\phi(x):x\in C \rbrace$, there is a unique embedding $\phi': (\Delta,\chi_{\Delta})\rightarrow (\Gamma^*,\chi^*)$ that extends $\phi$ with $\phi'(a)=a'$.
\end{enumerate}
\end{lemma}
\begin{proof} Let $a\in \Gamma'^{<0}$ with $[a]\notin [\Gamma]$ and $b=\chi'(a)\in \Gamma$. We define $(\Delta,\chi_{\Delta})=(\Gamma+ \mathbb{Q}a,\chi'_{a})$ where $\chi'_{a}$ is the restriction of $\chi'$ to $\Gamma+ \mathbb{Q}a$. As $\chi(\Gamma)=\chi'(\Gamma)$ then $\chi_{\Delta}(\Delta)=\chi(\Gamma)$, so $\chi_{\Delta}(\Delta)$ is a centripetal divisible precontraction group. 
\end{proof}

\subsubsection*{Case 3. $(\Gamma,\chi)\subseteq (\Gamma',\chi')$ with $\chi(\Gamma^{<0})\neq \chi'(\Gamma'^{<0})$.}
As we saw above if $(\Gamma,\chi)\subseteq (\Gamma',\chi')$ are model of $T_{pdg}$  and $\chi(\Gamma^{<0})\neq \chi'(\Gamma'^{<0})$, we have two cases. First, we can have that there is $b\in \chi'(\Gamma'^{<0})$ such that $b>\chi(\Gamma^{<0})$. So we want to extend $(\Gamma,\chi)$ to a model $(\Delta,\chi_{\Delta})$ of $T_{pdg}$ in which $b\in \chi_{\Delta}(\Delta^{<0})$. To do that, we can observe that if $b>\chi(\Gamma^{<0})$, then $\chi'^k(b)>\chi(\Gamma^{<0})$ for any integer $k$, where $\chi^{n+1}(b)=\chi(\chi^{n}(b))$, $\chi^0(b)=b$ and $\chi^{-n}(b)=c$ means that $\chi^n(c)=b$. Thus, to define the model $(\Delta,\chi_{\Delta})$ we need to add a copy of $\mathbb{Z}$ at the end of $\Gamma^{<0}$. Specifically, we have:\\

\begin{lemma} \label{kul3}
Let $(\Gamma,\chi)\subseteq (\Gamma',\chi')$ be divisible centripetal precontraction groups and $(b_n)_{n\geq 0}$ a family in $\chi(\Gamma'^{<0})$ such that $b_{n+1}=\chi'(b_n)$ and $b_n>\chi(\Gamma)$ for all $n\geq 0$, then there is a divisible centripetal precontraction group $(\Gamma'',\chi'')$ such that:
\begin{enumerate}
\item [$(1)$] $(\Gamma,\chi)\subset (\Gamma'',\chi'') \subseteq (\Gamma',\chi')$ with $b_n\in \chi''(\Gamma''^{<0})$ for $n>1$, and
\item  [$(2)$] for any embedding $\phi$ of $(\Gamma,\chi)$ into a divisible centripetal precontraction group $(\Gamma^*,\chi^*)$ and any family $(b_n')_{n\geq 0}$ in $\chi^*(\Gamma^{*<0})$ such that $b_{n+1}'=\chi^*(b_n')$ and $b_n>\phi(\chi(\Gamma^{<0}))$ for $n\geq 0$, there is a unique embedding $\phi': (\Gamma'',\chi'')\rightarrow (\Gamma^*,\chi^*)$ which extends $\phi$ and such that $\phi'(b_n)=b_n'$ for all $n$.

\end{enumerate}

\end{lemma}
\begin{proof} Let $((\Gamma_i,\chi_i)_{i\geq 0})$ be the family given by $\Gamma_0=\Gamma$, $\Gamma_{n+1}=\Gamma_n+\mathbb{Q}b_n$ and $\chi_n$ the restriction of $\chi'$ to $\Gamma_n$. By lemma \ref{kul2} $(\Gamma_n, \chi_n)$ is a divisible precontraction group for each $n$ and since $(\Gamma_n,\chi_n)\subseteq (\Gamma_{n+1},\chi_n)$ and $\chi'(b_n)=b_{n+1}$ then $(\Gamma'',\chi'')=\cup_{i\geq 0} (\Gamma_i,\chi_i)\subseteq (\Gamma',\chi')$ is a divisible centripetal precontraction group which extends $(\Gamma,\chi)$.
Now, by induction if we assume that $\phi_n:(\Gamma_n,\chi_n)\rightarrow (\Gamma^*,\chi^*)$ is an embedding such that $\phi_n(b_i)=b_i'$ for $i\in \lbrace 0,1,...,b_{n-1}\rbrace$, then by lemma \ref{kul2} there is a unique embedding 
\[\phi_{n+1}:(\Gamma_{n+1},\chi_{n+1})\rightarrow (\Gamma^*,\chi^*)\] 
which extends $\phi_n$ and such that $\phi_{n+1}(b_{n})=b_n'$. Thus, there is a unique embedding 
\[\phi'=\cup_{i\geq 0}\phi_i:(\Gamma'',\chi'') \rightarrow (\Gamma^*,\chi^*)\] which satisfies the required properties.
\end{proof}
\noindent
Now, we use the above lemma to include the predecessors of the element $b_0$ of the family:\\

\begin{lemma}\label{kul4}
Let $(\Gamma,\chi)\subseteq (\Gamma',\chi')$ be divisible centripetal precontraction groups and $(b_k)_{k\in \mathbb{Z}}$ a family in $\chi(\Gamma'^{<0})$ such that $b_{k+1}=\chi'(b_k)$ and $b_k>\chi(\Gamma)$ for all $k\in \mathbb{Z}$, then there is a divisible centripetal precontraction group $(\Delta,\chi_{\Delta})$ such that:
\begin{enumerate}
\item [$(1)$] $(\Gamma,\chi)\subset (\Delta,\chi_{\Delta}) \subseteq (\Gamma',\chi')$ with $b_k\in \chi_{\Delta}(\Delta^{<0})$ for all $k\in \mathbb{Z}$, and

\item  [$(2)$] for any embedding $\phi$ of $(\Gamma,\chi)$ into a divisible centripetal precontraction group $(\Gamma^*,\chi^*)$ and any family $(b_k')_{k\in \mathbb{Z}}$ in $\chi^*(\Gamma^{*<0})$ such that $b_{k+1}'=\chi^*(b_k')$ and $b_k>\phi(\chi(\Gamma^{<0}))$ for $k\in \mathbb{Z}$, there is a unique embedding $\phi': (\Delta,\chi_{\Delta})\rightarrow (\Gamma^*,\chi^*)$ which extends $\phi$ and such that $\phi'(b_k)=b_k'$ for all $k$.

\end{enumerate}
\end{lemma}
 \begin{proof} First for each $n\in \mathbb{N}$ we define the family $(a_i^n)_{i\geq 0}$ where  $a_{i}^n=b_{-n+i}$ for $i\geq 0$. Clearly, we have that $a_{i+1}^n=\chi'(a_i^n)$ and $a_{i+1}^{n+1}=a_{i}^n$. Now, using the lemma \ref{kul3} for each family $(a_i^n)_{i\geq 0}$ we obtain a divisible centripetal precontraction group $(\Gamma''_n,\chi''_n)$ such that $a_{i+1}^n\in \chi''_n((\Gamma''_n)^{<0})$ and $\chi''_n(a_i^n)=a_{i+1}^{n}$ and a unique embedding $\psi_n:(\Gamma''_n,\chi''_n)\rightarrow (\Gamma''_{n+1},\chi''_{n+1})$ such that $\psi_n(a_i^n)=a_{i+1}^{n+1}$. Thus we obtain the increasing chain
\[(\Gamma,\chi)\subset (\Gamma''_0,\chi''_0)\subset (\Gamma''_1,\chi''_1) \subset...\]
and we define $(\Delta,\chi_{\Delta})=\cup_{n\geq 0}(\Gamma''_{n},\chi''_{n})$.

Now, if $\phi:(\Gamma,\chi)\rightarrow \chi^*(\Gamma^{*<0})$ is an embedding with $(b_k')_{k\in \mathbb{Z}}$ a family in $\chi^*(\Gamma^{*<0})$ such that $b_{k+1}'=\chi^*(b_k')$ and $b_k>\phi(\chi(\Gamma^{<0}))$ for $k\in \mathbb{Z}$, then by lemma \ref{kul3} there is a unique embedding $$\phi_n:(\Gamma''_n,\chi''_n)\rightarrow (\Gamma^*,\chi^*)$$ 
\noindent
that extends $\phi$ and such that $\phi_n(a_{i}^n)=\phi_n(b_{-n+i})=b'_{-n+i}$. Moreover, $\phi_{n}\subseteq \phi_{n+1}$ because 
\[\phi_{n+1}(a_i^n)=\phi_{n+1}(a_{i+1}^{n+1})=b'_{-(n+1)+i+1}=b'_{-n+i}=\phi_n(a_i^n).\]
Thus we have that $\phi'=\cup\phi_n$ is the unique embedding from $(\Delta,\chi_{\Delta})$ into $(\Gamma^*,\chi^*)$ that extends $\phi$ and such that $\phi'(b_k)=b'_k$ for all $k\in \mathbb{Z}$.

\end{proof}
\noindent
On the other hand, if $(\Gamma,\chi)\subseteq (\Gamma',\chi')$ are models of $T_{pdg}$, $\chi(\Gamma^{<0})\neq \chi'(\Gamma'^{<0})$ and there is a nonempty special cut $G$ in $\chi(\Gamma^{<0})$ and $b\in \chi(\Gamma'^{<0})\setminus \chi(\Gamma^{<0})$ such that  $G<b<\Gamma^{>G}$, then there is a family $(b_k)_{k\in \mathbb{Z}}$ in $\chi'(\Gamma'^{<0})$ such that $G<b_k<\Gamma^{>G}$ and $b_{k+1}=\chi'(b_k)$. So, in order to extend $(\Gamma,\chi)$ to a model $(\Delta,\chi_{\Delta})$ of $T_{pdg}$ in which $b\in \chi_{\Delta}(\Delta^{<0})$ we have to add a copy of $\mathbb{Z}$ between some specific elements of $\chi(\Gamma)$. \\ 

\begin{lemma} \label{case4}
Let $(\Gamma,\chi)\subseteq (\Gamma',\chi')$ be divisible centripetal precontraction groups, $G$ be a nonempty special cut in $\chi(\Gamma^{<0})$ and $(b_k)_{k\in \mathbb{Z}}$ be a family in $\chi'(\Gamma'^{<0})$ such that $G<b_k<\Gamma^{>G}$ with $b_{k+1}=\chi'(b_k)$, then there is a divisible centripetal precontraction group $(\Delta_G,\chi_{G})$ such that:
\begin{enumerate}
\item [$(1)$] $(\Gamma,\chi)\subset (\Delta_G,\chi_{G}) \subseteq (\Gamma',\chi')$ with $b_k\in \chi_G(\Delta_G^{<0})$ for all $k\in \mathbb{Z}$, and
\item [$(2)$]  for any embedding $\phi$ of $(\Gamma,\chi)$ into a divisible centripetal precontraction group $(\Gamma^*,\chi^*)$ and any family $(b_k')_{k\in \mathbb{Z}}$ in $\chi^*(\Gamma^{*<0})$ such that $b_{k+1}'=\chi^*(b_k')$ and $\phi(G)<b_k'<\phi(\Gamma^{>G})$, there is a unique embedding $\phi': (\Delta_G,\chi_{G})\rightarrow (\Gamma^*,\chi^*)$ which extends $\phi$ and such that $\phi'(b_k)=b_k'$ for all $k$.
\end{enumerate}
\end{lemma}
\begin{proof} It is enough to take $\Delta_G=\Gamma+\oplus_{k\in \mathbb{Z}}\mathbb{Q}b_k$ and $\chi_{G}$ the restriction of $\chi'$ to $\Delta_G$. 
\end{proof}
\noindent
Under the hypothesis of the above lemma, for any element $a\in \Gamma^{<0}\setminus G$ we have $b_k<a<0$ for all $k\in \mathbb{Z}$, and by item 2 of  lemma \ref{cont1}, we obtain
\[\chi'(b_k-a)=b_k.\]
Thus, taking $a_k=b_k-a$ we can define $\Delta_G=\Gamma+\oplus_{k\in \mathbb{Z}}\mathbb{Q}a_k$ and $\chi_{G}$ the restriction of $\chi'$ to $\Delta_G$, with $b_k\in \chi_{G}(\Delta_{G}^{<0})$.

\subsection{Model completeness of $T_{pdg}$}

To prove the model completeness of $T_{pdg}$ we use the following result (see \cite[Corollary B.10.4.]{Van6}):\\

\begin{lemma} The following are equivalent:
\begin{enumerate}
\item [$(1)$] $\Sigma$ is model complete;
\item [$(2)$] for all models $M, N$ of $\Sigma$ with $M\subseteq N$ and every elementary extension $M^*$ of $M$ that is $k$-saturated for some $k>\card(N)$, there is an embedding $N\rightarrow M^*$ that extends the natural inclusion $M\rightarrow M^*$.\\

\end{enumerate}
\end{lemma}

\begin{remark} Let $M, N, M^*$ be models of $\Sigma$ where $M\subseteq N$, $M\preceq M^*$ and $M^*$ is  $k$-saturated for some $k>\card(N)$. If we want to show that $\Sigma$ is model complete, by the last lemma and Zorn's lemma, it is enough to show that there is a substructure $K$ of $N$ that properly contains $M$, is model of $\Sigma$ and embeds over $M$ in $M^*$.
\end{remark}
\noindent
Under such observation, the model completeness of $T_{pdg}$ is a consequence of the following theorem:\\

\begin{theorem}\label{modc} Let $(\Gamma,\chi), (\Gamma',\chi')$ and $(\Gamma^*,\chi^*)$ be models of $T_{pdg}$, such that $(\Gamma,\chi)\subseteq (\Gamma',\chi')$ and $(\Gamma^*,\chi^*)$ is a $k$-saturated elementary extension of $(\Gamma,\chi)$, with $k>\card(\Gamma')$. Then there is a  submodel $(\Delta,\chi_{\Delta})$ of  $(\Gamma',\chi')$ which properly extends $(\Gamma,\chi)$ such that $(\Delta,\chi_{\Delta})$ embeds over $(\Gamma,\chi)$ in $(\Gamma^*,\chi*)$.

\end{theorem}
\begin{proof}
We call $\phi$ the embedding of $(\Gamma,\chi)$ into $(\Gamma^*,\chi*)$ and  just consider the following cases:

\begin{enumerate}
\item [$(1)$] $[\Gamma]=[\Gamma']$: By lemma \ref{case1} it is enough to take $(\Delta,\chi_{\Delta})=(\Gamma',\chi')$.

\item [$(2)$] $[\Gamma]\neq [\Gamma']$ and $\chi(\Gamma)=\chi'(\Gamma)$: By hypothesis there is an element $a\in \Gamma'^{<0}$ such that $[a]\notin [\Gamma]$, and by lemma \ref{case2} there is a model $(\Delta,\chi_{\Delta})\subseteq (\Gamma',\chi')$ of $T_{pdg}$ that properly extends $(\Gamma,\chi)$. By saturation we can extend the embedding $\phi:(\Gamma,\chi)\rightarrow (\Gamma^*,\chi*)$ to an embedding $\phi':(\Delta,\chi_{\Delta})\rightarrow (\Gamma^*,\chi*)$.

\item [$(3)$] $\chi(\Gamma^{<0})\neq\chi'(\Gamma^{<0})$ and there is $b\in\chi(\Gamma'^{<0})$ such that $b>\chi(\Gamma^{<0})$: If we define the family $(b_k)_{k\in \mathbb{Z}}$ of $\chi(\Gamma'^{<0})$ by $b_0=b$, $b_{k+1}=\chi'(b_k)$ for $k>0$ and $b_{k-1}$ as the unique element of $\chi(\Gamma'^{<0})$ such that $\chi(b_{k-1})=b_k$ for $k<0$, then by lemma \ref{kul4} there is a model $(\Delta,\chi_{\Delta})\subseteq (\Gamma',\chi')$  of $T_{pdg}$ which properly extends $(\Gamma,\chi)$ and such that $b_k\in \chi_{\Delta}(\Delta^{<0})$.  

Using the saturation of $(\Gamma^*,\chi*)$ we can find a family $(b'_k)_{k\in \mathbb{Z}}$ in  $(\Gamma^*,\chi*)$ such that $b'_k>\phi(\chi(\Gamma^{<0}))$ for all $k\in \mathbb{Z}$ and $b_{k+1}=\chi^*(b_k)$. Thus, again by lemma \ref{kul4} there is a unique embedding 
\[\phi':(\Delta,\chi_{\Delta})\rightarrow (\Gamma^*,\chi*)\] 
that extends $\phi$ and such that $\phi'(b_k)=b'_k$.

\item [$(4)$] There is $b\in \chi'(\Gamma'^{<0})\setminus \chi(\Gamma)$ such that $b$ realize a special cut in $\chi(\Gamma^{<0})$: We define the set 
\[G=\lbrace a\in \chi(\Gamma^{<0}:a<b \rbrace.\]
Since the models of $T_{pdg}$ are centripetal precontraction groups then we have that $\chi(G)\subseteq G$ and by axioms $(3)$ and $(4)$ there is a family $(b_k)_{k\in \mathbb{Z}}$ in $\chi'(\Gamma'^{<0})$ such that $G<b_k<\Gamma^{>G}$, $b_0=b$, $b_{k+1}=\chi'(b_k)$ for $k>0$ and $b_{k-1}$ is the unique element of   $\chi'(\Gamma'^{<0})$ such that $\chi(b_{k-1})=b_k$ for $k<0$ then by lemma \ref{case4} there is a model $(\Delta,\chi_{\Delta})=(\Delta_G,\chi_G)\subseteq (\Gamma',\chi')$  of $T_{pdg}$ which properly extends $(\Gamma,\chi)$ and such that $b_k\in \chi_{\Delta}(\Delta^{<0})$. 

By saturation there is a family  $(b'_k)_{k\in \mathbb{Z}}$ in $(\Gamma^*,\chi*)$ such that $\phi(G)<b'_k<\phi(\Gamma^{>G})$, $b_{k+1}=\chi^*(b_k)$ for all $k\in \mathbb{Z}$, and again by lemma \ref{case4} there is a unique embedding $\phi':(\Delta,\chi_{\Delta})\rightarrow (\Gamma^*,\chi*)$ that extends $\phi$ and such that $\phi'(b_k)=b'_k$.

\end{enumerate}

\end{proof}

\begin{corollary} $T_{pdg}$ is model complete.

\end{corollary}
\noindent
Now, we can observe that the model  $(\oplus_i\mathbb{Q}e_i, \chi_{\mathbb{Q}})$ of $T_{pdg}$ defined in the first example of section \ref{sectcd} embeds in any model $(\Gamma, \chi')$ of $T_{pdg}$, since we can take any element $b\in \chi'(\Gamma^{<0})$, define the family $(b_n)_{n>0}$ such that $b_1=b$ and $b_{n+1}=\chi'(b_n)$, and identify the element $-e_n$ with the element $b_n$ for all $n\geq 1$. Thus we obtain that $T_{pdg}$ has a prime model and:\\

\begin{corollary} $T_{pdg}$ is complete.
\end{corollary}

\subsection{Quantifier elimination of $T_{pdg*}$}

Expanding the language $L_{pdg}$ to $L_{pdg*}=L_{pdg}\cup \lbrace\infty, \chi^{-1}, \delta_1, \delta_2, \delta_3,... \rbrace$, where $\infty$ is a constant symbol, $\chi^{-1}$ and $\delta_n$ for $n>0$ are unary function symbols, each model $(\Gamma, \chi)$ of $T_{pdg}$ can be seen as a $L_{pdg*}$-structure with underlying set $\Gamma_{\infty}=\Gamma\cup\lbrace \infty\rbrace$ in which:

\begin{itemize}
\item $\infty$ is such that $\infty\notin \Gamma$, $\infty+\infty=\chi(\infty)=-\infty=\infty$ and for all $x\in \Gamma$ we have $x+\infty=\infty$, and 
\item  we interpret $\delta_n$ as division by $n$ and $\chi^{-1}$ as a function from $\Gamma_{\infty}$ to $\Gamma_{\infty}$ such that its restriction 
\[\chi^{-1}:\chi(\Gamma^{<0})^{>c}\rightarrow\chi(\Gamma^{<0})\] 
is the inverse of $\chi:\chi(\Gamma^{<0})\rightarrow \chi(\Gamma^{<0})^{>c}$, $\chi^{-1}(0)=0, \chi^{-1}(c)=\infty$ and $\chi^{-1}(a)=\infty$ for all $a$ in $\Gamma^{\neq 0}_{\infty} \setminus \chi(\Gamma)$.
\end{itemize}
Thus, we define the theory $T_{pdg*}$ as the $L_{pdg*}$-theory whose models are the  expansion of models of $T_{pdg}$.

Now, we observe that each $L_{pdg*}$-substructure of a model of $T_{pdg*}$ has a $T_{pdg*}$-closure in the following sense:\\

\begin{lemma} Let $(\Gamma,\chi)$ be a model of $T_{pdg*}$ and $(\Gamma_0,\chi_{0})$ be a $L_{pdg*}$-substructure of $(\Gamma,\chi)$. There is a model $(\Gamma',\chi')$ of $T_{pdg*}$ such that
\begin{enumerate}
\item [$(1)$] $(\Gamma',\chi')\subseteq (\Gamma,\chi)$, and
\item [$(2)$] $(\Gamma',\chi')$ can be embedded over $(\Gamma_0,\chi_0)$ into every model of $T_{pdg*}$ which extends $(\Gamma_0,\chi_0)$.
\end{enumerate}
\end{lemma}
\begin{proof}
If there is $a\in \Gamma_0$ such that $\chi(a)=c$, then in fact $(\Gamma_0,\chi_0)$ is a model of $T_{pdg*}$ and we finish. Otherwise, there is $a\in \Gamma^{<0}$ such that $\chi(a)=c$, so we define $\Gamma'$ as the divisible ordered abelian group generated by $\Gamma_0\cup\lbrace a\rbrace$, and $\chi'=\chi|_{\Gamma'}$. Thus, $(\Gamma',\chi')$ is a model of $T_{pdg*}$.

Finally, given any model $(\Gamma^*,\chi^*)$ of $T_{pdg*}$ which extends $(\Gamma_0,\chi_0)$, there is $b\in \Gamma^*$ such that $\chi(b)=c$. We see that $a$ and $b$ have the same type over $\Gamma_0$. Thus, we define the embedding $\phi:(\Gamma',\chi')\rightarrow (\Gamma^*,\chi^*)$ as $\phi(\Gamma_0)=\Gamma_0$ and $\phi(a)=b$. 

\end{proof}
\noindent
As a consequence of this lemma and mimicking the proof of the theorem \ref{modc}, but considering $L_{pdg*}$-structures instead of $L_{pdg}$-structures, we can prove that the $L_{pdg*}$-theory $T_{pdg*}$ has quantifier elimination.

\subsection{Definable subsets of $\chi(\Gamma^{<0})$}

In this section we mimic the study made by Gehret in [2] about some definable sets in the asymptotic couple of $\mathbb{T}_{\log}$ and show that given a model $(\Gamma, \chi)$ of $T_{pdg*}$, each definable subset of $\chi(\Gamma^{<0})$ is a finite union of intervals in $\chi(\Gamma^{<0})$ and singletons. Specifically, to prove such result we will use a special kind of functions called $\chi$-functions\footnote{The notion of $\chi$-function used here was inspired in the notion of $\chi$-polynomial defined in \cite{Kuh3}}.

Now, for any element $a\in \chi(\Gamma^{<0})$ and integer $k<0$ we put $\chi^k(x)=(\chi^{-1})^{-k}(x)$ and $\chi^0(x)=x$.\\

\begin{definition}
We say that a function $G:\chi(\Gamma^{<0})\rightarrow \Gamma$ is a\textit{ $\chi$-function}\footnote{The notion of $\chi$-function used here was inspired in the notion of $\chi$-polynomial defined in \cite{Kuh3}} if it is constant or 
\[G(x)=\sum\limits_{i=1}^n q_i\chi^{k_i}(x)-\alpha\] 
for some $n>0$,  $k_1<k_2<...<k_n$ in $\mathbb{Z}$, $q_1,..., q_n\in \mathbb{Q}^{\neq 0}$ and $\alpha\in \Gamma$.\\

\end{definition}

Since for each $k\in \mathbb{Z}^{<0}$, the $\chi$-function $\chi^k(x)$ has image $\infty$ for $x<\chi^{-k}(c)$ with $x\in \chi(\Gamma^{<0})$, and it is injective and strictly increasing in $\chi(\Gamma^{<0})_k=\lbrace x\in \chi(\Gamma^{<0}):x\geq \chi^{-k}(c)\rbrace$, then if for any $\chi$-function  $G(x)=\sum\limits_{i=1}^n q_i\chi^{k_i}(x)+\alpha$ we define 

\[\Dom_G=
\begin{cases}
\chi(\Gamma^{<0})_{k_1} \text{ if } k_1<0\\
\chi(\Gamma^{<0}) \text{       if } k_1\geq 0\\
\end{cases} 
\]
then we have:

\begin{lemma} Let $G:\chi(\Gamma^{<0})\rightarrow \Gamma$ be the $\chi$-function given by $G(x)=\sum\limits_{i=1}^n q_i\chi^{k_i}(x)+\alpha$, then
\begin{enumerate}
\item [$(1)$] $G(a)=\infty$ for any $a\in \chi(\Gamma^{<0}) \setminus \Dom_G$.
\item [$(2)$] $G(x)$ is injective on $\Dom_G$.
\item [$(3)$] If $q_1>0$ then $G(x)$ is strictly increasing on $\Dom_G$, and if $q_1<0$ then $G(x)$ is strictly decreasing on  $\Dom_G$.
\end{enumerate}
\end{lemma}
\begin{proof}
\begin{enumerate}
\item [$(1)$] If $k_1<0$ then $\chi^{k_1}(a)=\infty$ for all $a<\chi^{-k_n}(c)$. Now, if $k_n>0$, then the proof is immediate.
\item [$(2)$] If $x\in \Dom_G$ then $\chi^{k_1}(x)<\chi^{k_2}(x) <...<\chi^{k_n}(x)$. So, if $y,x\in \Dom_G\subseteq\chi(\Gamma^{<0})$ with $y\neq x$, then $\chi^{k_i}(y)\neq \chi^{k_i}(x)$ for all $1\leq i\leq n$, and by lemma \ref{indep} we have that $G(x)\neq G(y)$.

\item [$(3)$] If $a,b\in \Dom_G$ with $a<b$, then $[a]<[b]$, $\chi(a)<\chi(b)$ and by lemma \ref{cont1} $\chi(a-b)=\chi(a)$. Thus, for all $i,j\in \mathbb{Z}$ with $i<j$ we have that
\[[\chi^i(b)-\chi^i(a)]>[\chi^j(b)-\chi^j(a)]\]
and then 
\[[\chi^{k_1}(b)-\chi^{k_1}(a)]>[\chi^{k_2}(b)-\chi^{k_2}(a)]>...>[\chi^{k_n}(b)-\chi^{k_n}(a)].\]
So, since $\chi^{k_1}(b)>\chi^{k_1}(a)$, we have that $G(b)-G(a)=\sum\limits_{i=1}^nq_i(\chi^{k_i}(b)-\chi^{k_i}(a))>0$ if and only if $q_1>0$.
\end{enumerate}

\end{proof}
\noindent
Since by lemma \ref{indep} we know that $\chi(\Gamma^{<0})$ is a linearly independent subset of $\Gamma$ as $\mathbb{Q}$-vector space, then depending on the constant value $\alpha$  we observe how many images has the restriction of the $\chi$-function 
\[G(x)=\sum\limits_{i=1}^n q_i\chi^{K_i}(x)-\alpha\] 
to $\Dom_G$ in $\chi(\Gamma^{<0})$:

\begin{lemma} \label{teo0} Given the $\chi$-function $G(x)=\sum\limits_{i=1}^n q_i\chi^{K_i}(x)+\alpha$ then we have one of the following possibilities:
\begin{enumerate}
\item [$(1)$] $\alpha=0$, $n=1$, $q_1=1$ and $G(\chi(\Gamma^{<0}))\subseteq \chi(\Gamma^{<0})$, or
\item [$(2)$] $\card(G(\Dom_G)\cap  \chi(\Gamma^{<0}))\leq 2$.
\end{enumerate}
\end{lemma}
\begin{proof}
Considering the element $\alpha$ we have two main cases: $\alpha$ does not belongs to $\Span_{\mathbb{Q}}\chi(\Gamma^{<0})$ or $\alpha$ belong to $\Span_{\mathbb{Q}}\chi(\Gamma^{<0})$. In the first case, $G(x)\notin \chi(\Gamma^{<0})$ for all $x\in \Dom_G$. In the second case we can assume that for some natural $m>0$ there are $r_1,r_2,...,r_m\in \mathbb{Q}$ and $a_1, a_2,...,a_m\in \chi(\Gamma^{<0})$ with $a_1<a_2<...<a_n$ such that $\alpha=r_1a_1+r_2a_2+...+r_ma_m$. Clearly, if $x\in \Dom_G$ then $G(x)\in \chi(\Gamma^{<0})$ if and only if $G(x)=\chi^{k_h}(x)$ for some $1\leq h\leq n$ or $G(x)=a_s$ for some $1\leq s\leq m$, which is possible only if all components except one of $G(x)$ are canceled. We analyze the possible cases:
\begin{itemize}
\item If $m=0$, i.e $\alpha=0$ and $n=1$, $q_1=1$ then for all $x\in \Dom_G$ we have $G(x)=\chi^{k_1}(x)\in \chi(\Gamma^{<0})$. 
\item If $|m-n|>1$ then for each element $x$ of $\Dom_G$ the value of $G(x)$ is a linear combination of at least two elements of $\chi(\Gamma^{<0})$. Thus, $G(\Dom_G)\cap \chi(\Gamma^{<0})=\emptyset$.
\item If $m=n+1$, then $G(x)$ belongs $\chi(\Gamma^{<0})$ only if $\card(\lbrace a_1, a_2,...,a_m\rbrace \cap \lbrace\chi^{k_i}(x): i\in \lbrace 1,2,...,n\rbrace\rbrace)=n$. Thus if $G(x)\in \chi(\Gamma^{<0})$  we have only two possibilities or $\chi^{k_1}(x)=a_1$ or $\chi^(k_n)=a_m$. So, since $G$ is injective on $\Dom_G$ then $\card(G(\Dom_G)\cap  \chi(\Gamma^{<0}))\leq 2$.
\item If $m=n$, then  $G(x)$ belongs $\chi(\Gamma^{<0})$ only if $\card(\lbrace a_1, a_2,...,a_m\rbrace \cap \lbrace\chi^{k_i}(x): i\in \lbrace 1,2,...,n \rbrace \rbrace)=n$ or equivalent $\chi^{k_i}(x)=a_i$ for all $1\leq i\leq n$. Since $G$ is injective on $\Dom_G$ then \[\card(G(\Dom_G)\cap  \chi(\Gamma^{<0}))=1.\]
\item If $n=m+1$, then analysis is similar to the case $m=n+1$.
\end{itemize}

\end{proof}
\noindent
Clearly if $G(x)$ and $H(x)$ are two $\chi$-functions then $G(x)+H(x)$, $G(x)-H(x)$ and $\delta_n(G(x))$ for all $n>0$ are again $\chi$-functions. Thus\\

\begin{lemma} \label{teo1} The set of $\chi$-functions is closed under $+, -, \delta_n$.
\end{lemma}
\noindent
On the other hand, although the composition $\chi(G(x))$ of $\chi$ and a $\chi$-function $G(x)$ is not necessarily a $\chi$-function, we can prove that $\chi(G(x))$ is given piecewise by $\chi$-functions (lemma \ref{teo2}), which means that there are $a_1, a_2, ..., a_n\in \chi(\Gamma^{0})\cup\lbrace 0 \rbrace$ with $c=a_1<a_2<...<a_n=0$ such that for any $i\in \lbrace 1, 2, ..., n-1\rbrace$ the restriction of $\chi(G(x))$ to $[a_i,a_{i+1})_{\chi}$ is given by a $\chi$-function.

To prove this, we first observe that by lemma \ref{indep}, for any element $\theta=\sum\limits_{i=1}^nq_ia_i$ of $\Gamma$ where $q_1,q_2,...,q_n\in \mathbb{Q}^{\neq 0}$ and $a_1, a_2,...,a_n\in \chi(\Gamma^{<0})$ with $a_1<a_2<...<a_n$, we have that $\chi(\theta)=\chi(a_1)$ if $q_1>0$ and $\chi(\theta)=-\chi(a_1)$ if $q_1<0$. Thus we have:\\

\begin{lemma} \label{teo2} Let $G(x)$ be a $\chi$-function. Then $\chi(G(x))$ is given piecewise by $\chi$-functions.
\end{lemma}
\begin{proof}
If $G(x)$ is constant, then $\chi(G(x))$ is also a constant, which means that $\chi(G(x))$ is a $\chi$-function. And if $G(x)=\sum\limits_{i=1}^n q_i\chi^{k_i}(x)+\alpha$ then  clearly, for all $x\in \chi(\Gamma^{<0})\setminus \Dom_G$ we have $\chi(G(x))=\chi(\infty)=\infty$ which is constant. So, from now on $G(x)$ will be a $\chi$-function of the form $G(x)=\sum\limits_{i=1}^n q_i\chi^{k_i}(x)+\alpha$ and we will focus on the values of $G$ on $\Dom_G$.

If $\alpha=0$ by the above lemma $\chi(G(x))=\sign(q_1)\chi(\chi^{k_1}(x))$ for all $x\in \Dom_G$. Putting now $\alpha\neq 0$ and $\theta(x)=\sum\limits_{i=1}^n q_i\chi^{K_i}(x)$ we have $\chi(G(x))=\chi(\theta(x)+\alpha)$.

Without loss of generality we can assume $q_1>0$. Thus $\chi(\theta(x))=\chi(\chi^{k_1}(x))$ for all $x\in \Dom_G$, and there is a unique $x_0\in \chi(\Gamma^{<0})$ such that $|\chi(\alpha)|=|x_0|$.  Thus we have two possibilities:

\begin{enumerate}
\item [$(1)$] $|\chi(\theta(x))|\neq |x_0|$ for all $x\in \Dom_G$. If $\chi(\alpha)=x_0$ then either $x_0<\chi(\theta(x))$ for all $x\in \Dom_G$ and $\chi(G(x))=x_0$ for all $x\in \Dom_G$, or there is a unique $x_1\in \Dom_G$ such that \[\chi(\theta(x_1))<x_0<\chi(\theta(\chi(x_1))\]
and 
\[\chi(G(x))=
\begin{cases}
[\chi(\chi^{k_1}(x)) \text{ if } x<x_1\\
x_0 \text{       if } x>x_1\\
\end{cases} 
\]
Now, if $\chi(\alpha)=-x_0$ then $\chi(G(x))=\chi(\chi^{k_1}(x))$ for all $x\in \Dom_G$.
\item [$(2)$] There is a unique $x_1\in \Dom_G$ such that $|\chi(\theta(x_1))|= |x_0|$. We can see that $\chi(G(x))$ has the same behavior for all $x\neq x_1$ that in the previous case . However, if $x=x_1$ then we have the following cases: If $\chi(\alpha)=x_0$ then $\chi(G(x))=x_0$, but if $\chi(\alpha)=-x_0$ then:
Let $\alpha_1=\alpha+q_1\chi^{k_1}(x)$. If  $\chi(\alpha_1)=\chi(\chi^{k_1}(x))$ then $\chi(G(x))=\chi(\chi^{k_1}(x))$. In other case, we compare $|\chi(\chi^{k_2}(x))|$ with $|\chi(\alpha_1)|$. If $|\chi(\chi^{k_2}(x))|\neq |\chi(\alpha_1)|$ then the value of $G(x)$ is determined by the $\min\lbrace \sign(q_2)\chi(\chi^{k_2}(x)), \chi(\alpha_1) \rbrace$. If $|\chi(\chi^{k_2}(x))|= |\chi(\alpha_1)|$ then we have two cases, if $\sign(q_2)\chi(\chi^{k_2}(x))=\chi(\alpha_1)$ then $\chi(G(x))=\sign(q_2)\chi(\chi^{k_2}(x))$, but if not, then we define $\alpha_2=\alpha_1+q_2\chi^{k_2}(x)$ and repeat the analysis done for $\alpha_1$. This process is finite because in the possible last step we analyze $\alpha_n=\alpha_{n-1}+q_n\chi^{k_n}(x)$.
\end{enumerate}
In conclusion, for each $\chi$-function $G(x)=\sum\limits_{i=1}^n q_i\chi^{k_i}(x)+\alpha$, $\chi(G(x))$  is given piecewise by $\chi$-functions.\\

\end{proof}
\noindent
From lemmas \ref{teo0}, \ref{teo1} and \ref{teo2} we obtain:\\

\begin{proposition}\label{prop1} Let $t(x):\Gamma\rightarrow \Gamma$ be an $L_{pdg*}$-term and $G:\chi(\Gamma^{<0})\rightarrow\Gamma$ the restriction of $t$ to $\chi(\Gamma^{<0})$. Then $G$ is given piecewise by $\chi$-functions.
\end{proposition}
\begin{proof}
The proof follows from lemmas \ref{teo0}, \ref{teo1} and \ref{teo2} doing induction on the complexity of the $L_{pdg*}$-terms.
\end{proof}
\noindent
As a consequence of this proposition and the quantifier elimination in $T_{pdg*}$ we have:\\

\begin{corollary}\label{coro1} Every definable $A\subseteq \chi(\Gamma^{<0})$ is a finite union of intervals in $\chi(\Gamma^{<0})$ and singletons.\\

\end{corollary}

\begin{remark} For each model $(\Gamma,\chi)$ of $T_{pdg*}$, the definable set $\chi(\Gamma^{<0})$ is infinite and discrete, so $(\Gamma,\chi)$ is not weakly o-minimal.\\

\end{remark}
\noindent
Now, if we expand the language $L_{pdg*}$ by a new constant symbol $d$ and define the theory $T_{pdg**}$ as 
\[T_{pdg*}\cup\lbrace \chi(d)=c\rbrace\]
then $T_{pdg**}$ has quantifier elimination and a universal axiomatization. Thus, from proposition \ref{prop1} we have the following:

\begin{theorem} Let $G:\chi(\Gamma^{<0})\rightarrow\Gamma$ be a definable function. Then $G$ is given piecewise by $\chi$-functions. 
\end{theorem}

\begin{proof} Since $T_{pdg**}$ has quantifier elimination and has a universal axiomatization, then by corollary B.11.15 of \cite{Van6}  there are $L_{pdg**}$-terms $t_1(x),t_2(x),...,t_n(x)$ such that $G(x)=t_k(x)$ for $x\in \chi(\Gamma^{0})$ and some $k\in \lbrace 1,2,...,n\rbrace$. Thus, by proposition \ref{prop1} the restriction of $G(x)$ to 
\[\Dom_k=\lbrace x\in \chi(\Gamma^{<0}):G(x)=t_k(x)\rbrace\subseteq \chi(\Gamma^{0}),\]
is given piecewise by $\chi$-functions.

\end{proof}

\subsection{Simple extensions}

Let $\mathbb{M}=(M, \chi_M)$ be a monster model of $T_{pdg*}$ and $(\Gamma,\chi)$ a small submodel of $\mathbb{M}$. In this section we show that each simple extension $\Gamma\langle a\rangle$ for $a\in M\setminus \Gamma$ of $\Gamma$ is isomorphic to a specific extension of $\Gamma$ obtained utilizing the extensions given in lemmas \ref{kul4} and \ref{case4}.

To do that, first we will combine the lemmas \ref{kul4} and \ref{case4} to define extensions of $\Gamma$ which are built including many copies of $\mathbb{Z}$ in a specific and ordered way. Specifically, if $scut^{op}(\chi(\Gamma^{<0}))$ denote the linear ordered set of the elements $G\subseteq \chi(\Gamma^{<0})$ such that $\chi(\Gamma^{<0})\setminus G$ is a special cut of $\chi(\Gamma^{<0})$ and where $G_1\leq G_2$ in $scut^{op}(\chi(\Gamma^{<0}))$ if and only if $G_2\subseteq G_1$, then given an ordinal $\delta$ and an increasing function $f:\delta\rightarrow scut^{op}(\chi(\Gamma^{<0}))\setminus \lbrace \chi(\Gamma^{<0})\rbrace$, for each $f(\alpha)$ with $\alpha<\delta$ we want to include a specific copy of $\mathbb{Z}$ between $\chi(\Gamma^{<0})\setminus f(\alpha)$ and $f(\alpha)$. Moreover, if $\delta=\beta+1$, it may happen that $f(\beta)=\emptyset$, which means that we have to include a copy of $\mathbb{Z}$ at the end of $\chi(\Gamma^{<0})$. \\

\begin{lemma} Let $\delta$ be an ordinal. Given a increasing function $f:\delta\rightarrow scut^{op}(\chi(\Gamma^{<0}))\setminus \lbrace \chi(\Gamma^{<0})\rbrace$, there is a model $(\Gamma_f,\chi_f)$ of $T_{pdg}$ and a family $(b_{k,\rho})_{k\in \mathbb{Z},\rho<\delta}$ in $\chi(\Gamma_f^{<0})$ such that:

\begin{enumerate}
\item [$(1)$] $(\Gamma,\chi)\subset (\Gamma_f,\chi_{f})$,
\item [$(2)$] $\Gamma^{<f(\rho)}<b_{k,\rho}<f(\rho)$ and $\chi_f(b_{k,\rho})=b_{k+1,\rho}$ for all $k\in \mathbb{Z},$ and $\rho<\delta$,
\item [$(3)$] $b_{k_1,\rho_1}<b_{k_2,\rho_2}$ for all $k_1, k_2\in \mathbb{Z}$ and $\rho_1<\rho_2<\delta$, and
\item [$(4)$] for any embedding $\phi$ of $(\Gamma,\chi)$ into a model $(\Gamma^*,\chi^*)$ of $T_{pdg}$ and any family $(b^*_{n,k})_{n\in \mathbb{Z},k<\delta}$ in $\chi(\Gamma^{*<0})$such that $\phi(\Gamma^{<f(\rho)})<b_{k,\rho}<\phi(f(\rho))$ and $\chi^*(b^*_{k,\rho})=b^*_{k+1,\rho}$ for all $k\in \mathbb{Z},$ and $\rho<\delta$, and $b^*_{k_1,\rho_1}<b^*_{k_2,\rho_2}$ for all $k_1, k_2\in \mathbb{Z}$ and $\rho_1<\rho_2<\delta$, there is a unique embedding $\phi'$ from $(\Gamma_f,\chi_{f})$ into $(\Gamma^*,\chi^*)$ which extends $\phi$ and such that $\phi'(b_{k,\rho})=b^*_{k,\rho}$ for all $k\in \mathbb{Z}$ and $\rho<\delta$.
\end{enumerate}

\end{lemma}

\begin{proof} The proof is by induction on $\delta$ and we only have to observe that for the successor step, if $\delta=\beta+1$, then by inductive hypothesis there is an extension $(\Gamma_{f|\beta}, \chi_{f|\beta})$ of  $(\Gamma,\chi)$ corresponding to 
\[f|\beta:\beta\rightarrow scut^{op}(\chi(\Gamma^{<0}))\setminus \lbrace \chi(\Gamma^{<0})\rbrace\]
and $f(\beta)\in scut^{op}(\chi(\Gamma_{f|\beta}^{<0}))$.
\end{proof}
\noindent
Now, to study the simple extension $\Gamma\langle a\rangle$ of $\Gamma$ with $a\in M\setminus \Gamma$, we consider first if $(\Gamma\oplus \mathbb{Q}a)^{<0}$ is closed under $\chi$ and to do that we use   the set 
\[\Delta_{\Gamma}=\chi((\Gamma+\mathbb{Q}^{\neq 0}a)^{<0})=\lbrace \chi(x+qa): q\in \mathbb{Q}^{\neq 0} \text{, } x\in \Gamma \text{ and  } x+qa<0\rbrace.\]
Specifically, we have the following results:\\

\begin{lemma} 
\mbox{}
\begin{enumerate}
\item [$(1)$] For all $x\in M^{<0}$ and $y\in \Delta_{\Gamma}$ with $x<y$,  $x\in \Delta_{\Gamma}$ if and only if $x\in \chi(\Gamma^{<0})$.
\item [$(2)$] For all $x\in \chi(\Gamma^{<0})$ and $y\in \Delta_{\Gamma}\cap \chi(\Gamma^{<0})$, if $x<y$ then $x\in \Delta_{\Gamma}\cap \chi(\Gamma^{<0})$.
\item [$(3)$] $\card(\Delta_{\Gamma}\setminus \chi(\Gamma^{<0}))\leq 1$.
\item [$(4)$] If $\Delta_{\Gamma}\setminus \chi(\Gamma^{<0})=\lbrace b \rbrace$ with $b\in \chi(M^{<0})\setminus \chi(\Gamma^{<0})$, then $b$ realize the special cut 
\[(\Delta_{\Gamma}\cap \chi(\Gamma^{<0}), \chi(\Gamma^{<0})\setminus \Delta_{\Gamma})\]
in $\chi(\Gamma^{<0})$ 
\end{enumerate}

\end{lemma}

\begin{proof}
\begin{enumerate}

\item [$(1)$] Let $y=\chi(b+qa)$ for some $b\in \Gamma$ and $q\in \mathbb{Q}^{\neq 0}$. If $x\in \Delta_{\Gamma}$, then $x=\chi(d+ra)$ for some $d\in \Gamma$ and $r\in \mathbb{Q}^{\neq 0}$. Without loss of generality we can assume that $q,r>0$. Thus, $x=\chi(\dfrac{q}{r}d+qa)$ and since $x<y<0$ then
\[x=\chi(\dfrac{q}{r}d+qa)=\chi((\dfrac{q}{r}d+qa)-(b+qa))=\chi(\dfrac{q}{r}d-b)\in \chi(\Gamma^{<0}).\]
On the other hand, if $x\in \chi(\Gamma^{<0})$ then $x=\chi(d)$ for some $d\in \Gamma^{<0}$. Thus
\[x=\chi(d)=\chi(d+(b+qa))=\chi((d+b)+qa)\in \Delta_{\Gamma}.\]
If $q<0$ or $r<0$, the demonstration is similar.
\item [$(2)$] It follows from $(1)$.
\item [$(3)$] If we assume that there are $x,y\in \Delta_{\Gamma}\setminus\chi(\Gamma^{<0})$, with $x< y$, then since $x,y\in \Delta_{\Gamma}$ then by item $(1)$ we obtain that $x\in \chi(\Gamma^{<0})$, a contradiction. 
\item [$(4)$] It follows by items $(2)$ and $(3)$.
\end{enumerate}
\end{proof}
\noindent
As a consequence of the above, we have two possibilities $\chi(\Delta_{\Gamma})\subseteq \Delta_{\Gamma}$ or  $\chi(\Delta_{\Gamma})\setminus \Delta_{\Gamma}\neq \emptyset$. Hence it follows that:

\begin{corollary} Exactly one of the following is true:
\begin{enumerate}
\item [$(1)$] There is a nonempty special cut $B$ in $\chi(\Gamma^{<0})$ such that $\Delta_{\Gamma}=B$.
\item [$(2)$] There is $b\in \chi(\Gamma^{<0})$ such that $\Delta_{\Gamma}= (\chi(\Gamma^{<0}))^{\leq b}\subseteq \chi(\Gamma^{<0})$.
\item [$(3)$] There is a nonempty special cut $B$ in $\chi(\Gamma^{<0})$ and $b\in \chi(M^{<0})\setminus \chi(\Gamma^{<0})$ such that $B<b$, $b<(\chi(\Gamma^{<0})\setminus B)$ and $\Delta_{\Gamma}=B\cup \lbrace b\rbrace$
\end{enumerate}
\end{corollary}
\noindent
As a particular case, if $\Delta_{\Gamma}\subseteq \chi(\Gamma^{<0})$ then the ordered divisible abelian subgroup $\Gamma\oplus \mathbb{Q}a$ of $M$ is closed under $\chi$ and $\Gamma\langle a\rangle=(\Gamma\oplus \mathbb{Q}a,\chi)$. In general we have the following:\\

\begin{theorem} If $a\in M\setminus \Gamma$, then $\Gamma\langle a\rangle$ is isomorphic over $\Gamma$ to one of the following:

\begin{enumerate}
\item [$(1)$] $\Gamma_f$ for some increasing function $f:n\rightarrow scut^{op}(\Gamma)\setminus \lbrace \chi(\Gamma^{<0})\rbrace$ and some natural $n$.
\item [$(2)$] $\Gamma_f\oplus \mathbb{Q}a$ for some increasing function $f:n\rightarrow scut^{op}(\Gamma)\setminus \lbrace \chi(\Gamma^{<0})\rbrace$ and some natural $n$
\item [$(3)$]$\Gamma_f\oplus \mathbb{Q}a$ for some increasing function $f:\omega\rightarrow scut^{op}(\Gamma)\setminus \lbrace \chi(\Gamma^{<0})\rbrace$ 
\end{enumerate}
\end{theorem}
\begin{proof} The main idea of the proof is to construct by induction a chain $\Gamma_0\subseteq \Gamma_1\subseteq \Gamma_2\subseteq...\subseteq \Gamma\langle a\rangle$ of models of $T_{pdg*}$ in the model $\mathbb{M}$, each one isomorphic to $\Gamma_f$ for some increasing function 
\[f:n\rightarrow scut^{op}(\Gamma)\setminus \lbrace \chi(\Gamma^{<0})\rbrace.\]
To do that, we put first $\Gamma_0=\Gamma$. Clearly, $\Gamma_0$ is isomorphic to $\Gamma_f$ for $f:0\rightarrow scut^{op}(\Gamma)\setminus \lbrace \chi(\Gamma^{<0})\rbrace$. Assume we have built $\Gamma_n$ with $n\in \mathbb{N}$ and $\Gamma_n\cong \Gamma_f$ for some increasing $f:n\rightarrow scut^{op}(\Gamma)\setminus \lbrace \chi(\Gamma^{<0})\rbrace$ then we have two possibilities: 
\begin{enumerate}
\item [$(1)$] $\Gamma_n=\Gamma\langle a\rangle$, and then $\Gamma\langle a\rangle\cong \Gamma_f$.
\item [$(2)$] $a\notin \Gamma_n$. Thus we consider the set $\Delta_{\Gamma_n}$ for $\Gamma_n$, and we have other two cases:
\begin{itemize}
\item $\Delta_{\Gamma_n}\subseteq \chi(\Gamma_n)$. Thus, we put $\Gamma_{n+1}=\Gamma_n\oplus \mathbb{Q}a$. So, $\Gamma\langle a\rangle\cong \Gamma_f=\Gamma_{n+1}$ and $\Gamma_f\cong \Gamma_f\oplus \mathbb{Q}a$.
\item $\chi(\Delta_{\Gamma_n})\setminus \Delta_{\Gamma}$. Here, $\Delta_{\Gamma_n}=B\cup\lbrace b\rbrace$ for some special cut $B\subseteq \chi(\Gamma_{n}^{<0})$ and $b\in \chi(M^{<0})\setminus \chi(\Gamma_{n}^{<0})$ with $B<b<(\chi(\Gamma_n^{<0}\setminus B$. Thus, we define $\Gamma_{n+1}$ as the model of $T_{pdg}$ given by lemma \ref{case4} by including the copy of $\mathbb{Z}$ corresponding to $b$. Thus, there is $g:n+1\rightarrow scut^{op}(\Gamma)\setminus \lbrace \chi(\Gamma^{<0})\rbrace$ such that $\Gamma_{n+1}\cong \Gamma_g$.
\end{itemize}

\end{enumerate}
Now, if $\Gamma\langle a\rangle=\Gamma_n$ for some $n$ we have finish. Otherwise, we put $\Gamma\langle a\rangle=\cup_{n} F_n\oplus \mathbb{Q}a$. By construction, $\Gamma\langle  a\rangle\cong \Gamma_{f}\oplus \mathbb{Q}a$ for some increasing $f:\omega\rightarrow scut^{op}(\Gamma)\setminus \lbrace \chi(\Gamma^{<0})\rbrace$.

\end{proof}

\begin{example}
\begin{enumerate}
\item [$(1)$] Let $(\oplus_{i}\mathbb{Q}e_i,\chi_{\mathbb{Q}})\subseteq (\Gamma_{\log},\chi)$ be the model of $T_{pdg}$ considered in the first example of section \ref{sectcd}, $r\in \mathbb{R}^{<0}\setminus \mathbb{Q}$   and $a=re_m\in \Gamma_{\log}\setminus \oplus_{i}\mathbb{Q}e_i$, for some $m$. Since for each $b+qa\in  (\Gamma_{\log}+\mathbb{Q}^{\neq 0}a)^{<0}$ the entry $m$ never is $0$, then 
\[\Delta_{\Gamma_{\log}}=\chi(\Gamma_{\log}+\mathbb{Q}^{\neq 0}a)^{<0}=\lbrace-e_i:2\leq i\leq m\rbrace \subseteq \chi(\Gamma_{\log}^{<0}).\]
Hence, 
\[(\oplus_{i}\mathbb{Q}e_i,\chi_{\mathbb{Q}})\langle a \rangle=(\oplus_{i}\mathbb{Q}e_i\oplus\mathbb{Q}a,\chi')\subseteq (\Gamma_{\log},\chi)\]
where $\chi'$ is the restriction of $\chi$ to $\oplus_{i}\mathbb{Q}e_i\oplus\mathbb{Q}$.
\item [$(2)$] Let $(\Gamma,\chi)$ be a model of $T_{pdg}$ and $(\Gamma_{f},\chi_f)$ be a fixed extension of $(\Gamma,\chi)$ for some increasing function 
\[f:n\rightarrow scut^{op}(\Gamma)\setminus \lbrace \chi(\Gamma^{<0})\rbrace\]
with  $n\geq 1$. Let's take one element $a_j\in (\Span_{\mathbb{Q}}(b_{k,j})_{k\in \mathbb{Z}})^{\neq 0}$ for each $j<n$, where $(b_{k,j})_{k\in \mathbb{Z}}$ are the elements of the $j$-th copy of $\mathbb{Z}$ added to $\Gamma$ in $\Gamma_{f}$. Given $c\in \Gamma$ we define the element
\[a=c+\sum\limits_{j=0}^{n-1}a_j\in \Gamma_f.\]
Thus, $\Gamma\langle a \rangle=\Gamma_f$.

\end{enumerate}
\end{example}


\noindent
\textbf{Acknowledgements} The author thank Lou van den Dries and Xavier Caicedo for the helpful remarks and suggestions. Particularly, this paper was written with the support of the research fund of the faculty of sciences at the Universidad de los Andes, within the framework of the "Convocatoria 2018-1 para la financiaci\'on de Proyectos de Investigaci\'on y Participaci\'on en Eventos Acad\'emicos categoria Estudiantes de Doctorado" and the project INV-2017-26-1141.

%

\begin{thebibliography}{[1]}
  
\bibitem{Van6}%
M. Aschenbrenner, L. van den Dries, and J. van der Hoeven, 
Asymptotic Differential Algebra and Model Theory of Transseries, 
Ann. of Math. Stud, \textbf{195}, (2017).

\bibitem{Chang}%
C.C. Chang and H.J. Keisler,
Model theory, 
3rd ed., Studies in Logic and the Foundation of Mathematics, North-Holland, Amsterdam \textbf{73},(1990).

\bibitem{Danh}%
B. Dahn and P. Göring,
Notes on exponential-logarithmic terms, 
Fund. Math., \textbf{127}(1),45-50, (1987).

\bibitem{Ecalle}%
J. \'Ecalle,
Introduction aux Fonctions Analysables et Preuve Constructive de la Conjeture de Dulac, 
Actualit\'es Math\'ematiques, Hermann, Paris, (1992).

\bibitem{Ger1}%
A. Gehret, 
The asymptotic couple of the field of logarithmic transseries, 
J. of Alg., \textbf{470},1-36 (2017).

\bibitem{Hod}%
W. Hodges, 
Model theory, 
Encyclopedia of Mathematics and its Applications, Cambridge University Press, \textbf{42}, (1993).

\bibitem{Kuh2}%
F. Kuhlmann, 
Abelian groups with contractions I,
in: Proceedings of the Oberwolfach Conference on Abelian Groups 1993, Amer. Math. Soc. Contemporary Mathematics, \textbf{171},  217-241 (1994).
\bibitem{Kuh3}%
F. Kuhlmann, 
Abelian groups with contractions II: Weak o-minimality, 
in: Abelian groups and modules (proceedings of the Padova conference 1994), Mathematics and its Applications, Springer, Dordrecht, \textbf{343}, 323-342 (1995).

\bibitem{Kuh4}%
S. Kuhlmann,
On the structure of nonarchimedean exponential fields I, 
Arch. Math. Log., \textbf{34}, 145-182 (1995).

\bibitem{Kuh6}%
F. Kuhlmann and S. Kuhlmann,
On the structure of nonarchimedean exponential fields II, 
Comm. in Algebra, \textbf{22}, 5079-5103 (1994).

\bibitem{Kuh5}%
F. Kuhlmann and S. Kuhlmann, 
The exponential rank of nonarchimedean exponential fields, 
Delzell and Madden (eds): Real algebraic geometry and ordered structures, Contemp. Math. \textbf{253}, 181-201 (2000).


\bibitem{New}%
B. Newman,  
On the ordered division rings, 
Trans. of the Amer. Math. Soc., \textbf{66}(1), pp 202-252 (1949).



\end{thebibliography}
%

\textsc{Departamento de Matem\'aticas, Universidad de los Andes, Cra. 1. No. 18A-10, Bogotá, Colombia.}

\textit{E-mail address:} jl.angel76@uniandes.edu.co

\end{document}